\documentclass[a4paper,12pt]{amsart}
\usepackage[english]{babel}
\usepackage{amsmath, amsthm, amsfonts, amssymb}
\makeatletter
\@addtoreset{equation}{section}
\@addtoreset{thm}{section}
\@addtoreset{lem}{section}
\@addtoreset{expl}{section}
\@addtoreset{remk}{section}
\@addtoreset{nsl}{section}
\@addtoreset{defn}{section}
\makeatother
\newcommand{\ov}{\overline}

\newcommand{\vk}{\varkappa}
\newcommand{\ve}{\varepsilon}
\newcommand{\cA}{{\mathcal A}}
\newcommand{\cM}{{\mathcal M}}
\newcommand{\wt}{\widetilde}

\newcommand{\pt}{\partial}
\newcommand{\1}{1\!\!\,{\rm I}}
\newcommand{\mbR}{{\mathbb R}}

\newcommand{\mbN}{{\mathbb N}}
\theoremstyle{plain}
\newtheorem{thm}{Theorem}[section]
\newtheorem{lem}{Lemma}[section]

\theoremstyle{definition}
\newtheorem{defn}{Definition}[section]

\theoremstyle{remark}

\newcommand{\E}{\mathrm{E}}
\newcommand{\e}{\mathrm{e}}
\newcommand{\Law}{\mathrm{Law}}
\newcommand{\Var}{\mathrm{Var}}
\newcommand{\Prob}{\mathrm{P}}

\theoremstyle{plain}
\newtheorem{innercustomthm}{Theorem}
\newenvironment{customtheorem}[1]
  {\renewcommand\theinnercustomthm{#1}\innercustomthm}
  {\endinnercustomthm}
\makeatletter
\@namedef{subjclassname@2020}{%
  \textup{2020} Mathematics Subject Classification}
\makeatother

\begin{document}
\title
[On approximations of the point measures associated with the Brownian web\ldots]
{On approximations of the point measures associated with the Brownian web by means of the fractional step method and the discretization of the initial interval}

\author{A.A.Dorogovtsev}
\address{A.A.Dorogovtsev: Institute of Mathematics, National Academy of Sciences of Ukraine, Tereshchenkivska Str. 3, Kiev 01601, Ukraine; National Technical University of Ukraine ``Igor Sikorsky Kyiv Polytechnic Institute'', Institute of Physics and Technology, Peremogi avenue 37, Kiev 01601, Ukraine}
\email{andrey.dorogovtsev@gmail.com}

\author[M.B.Vovchanskii]{M.B.Vovchanskii}
\address{M.B.Vovchanskii: Institute of Mathematics, National Academy of Sciences of  Ukraine, Tereshchenkivska Str. 3, Kiev 01601, Ukraine}
\email{vovchansky.m@gmail.com}


\begin{abstract}
The rate of the weak convergence in the fractional step method for the Arratia flow is established in terms of the Wasserstein distance between the images of the Lebesque measure under the action of the flow. We introduce finite-dimensional  densities  describing sequences of collisions in the Arratia flow and derive an explicit expression for them. With the initial interval discretized, the convergence of the corresponding approximations of the point measure associated with the Arratia flow is discussed in terms of such densities. 
\end{abstract}

\maketitle
\section{Introduction}
In this article we consider point measures which are constructed from the Arratia flow and its approximations \cite{1,3,9,10}. Two types of discrete measures can be associated with a stochastic flow $\left\{X(u,t)\mid t\ge 0, u\in\mbR\right\}$ with coalescence on the real line: the first measure is the image of the Lebesque measure under the action of the flow 
$$
\mu_t = \lambda \circ \left(X(\cdot, t)\right)^{-1},
$$
and the second one is the counting measure defined by the rule 
$$
\nu_t(\Delta) = \left| X(\mbR, t) \cap \Delta \right|, \quad \Delta\in\mathcal{B}(\mbR).
$$
Both measures are supported on the same locally finite countable set. The structure of such random measures is studied in \cite{15, 16, 22, 24, 25}. In the first part of the article the Arratia flow with drift is considered. This flow consists of coalescing Brownian motions with diffusion $1$ and drift $a,$ where $a$ is a bounded Lipschitz continuous function. Such a stochastic flow was obtained in \cite{1} by applying the fractional step method \cite{11, kotelenez} to the Brownian web \cite{9, 10} and an ordinary differential equation driven by $a.$ Here the study of this approximation scheme is continued by discussing the speed of convergence of the images of the Lebesque measure. 

We start with recalling the fractional step method for the Brownian web proposed in \cite{1}. Let $a$ be a bounded Lipschitz continuous  function on the real line. Consider a sequence of partitions $\{0=t^{(n)}_0<\ldots< t^{(n)}_n=1\}$ of the interval $[0; 1]$ with the mesh size $\delta_n$ converging to $0.$ Define a family of transformations of $\mbR$
$$
\begin{cases}
d\cA_{s,t}(u)=a(\cA_{s,t}(u))dt,\\
\cA_{s,s}(u)=u, \ t\geq s.
\end{cases}
$$
Given a Brownian web $\left\{\Phi_{s,t}(u) \ | \ 0\leq s\leq t, u\in\mbR\right\}$ \cite{9, 10} one can consider  $\{\Phi_{s,t}\}_{0\leq s\leq t}$ as random mappings of $\mbR$ into itself. Put $\Delta^{(n)}_j=[t^{(n)}_j; t^{(n)}_{j+1}), j=\ov{0, n-1},$ and define, for  $u\in\mbR, t\in \Delta^{(n)}_j,$
\begin{align*}
\Phi^{(n)}_t(u)& =\Phi_{t^{(n)}_j, t}\circ\cA_{t^{(n)}_j, t^{(n)}_{j+1} }
\Big(\mathop{\circ}\limits^{j-1}_{l=0}
(\Phi_{t^{(n)}_l, t^{(n)}_{l+1}}\circ \cA_{t^{(n)}_l, t^{(n)}_{l+1}})
(u)
\Big), \\
\Phi^{(n)}_1(u) &=\lim_{t\to1-}\Phi^{(n)}_t(u).
\end{align*}
The sign $\circ$ stands for the composition of functions: $f\circ g=f(g).$ The main result of \cite{1} states that given $u_1,\ldots, u_m\in\mbR$
\begin{equation}
\label{eq:convergence}
(\Phi^{(n)}(u_1), \ldots, \Phi^{(n)}(u_m))\mathop{\Rightarrow}_{n\to\infty}
(\Phi^{a}(u_1), \ldots, \Phi^{a}(u_m))
\end{equation}
in the Skorokhod space $(D([0; 1]))^m,$ with $\{\Phi^a_s(u)|s\geq0, u\in\mbR\}$ being an Arratia flow with drift $a$ \cite[\S 7.3]{3}. It was proven in \cite[Proposition 1.5]{1} that the sequence in the left hand side of \eqref{eq:convergence} converges only weakly in contrast to the application of the fractional step method to ordinary SDEs \cite{11, kotelenez}. 

Let $\lambda$ be the Lebesque measure on $[0; 1].$ One can define images of $\lambda$ under the mappings $\Phi^a_t, \Phi^{(n)}_t:$
$$
\mu_t=\lambda\circ\left(\Phi^a_t\right)^{-1}, \ \mu^{(n)}_t=\lambda\circ\left(\Phi^{(n)}_t\right)^{-1}, \quad n\in\mbN.
$$
Such random measures along with associated point processes are central objects of the present paper, in the first part of which an estimate on the speed of the convergence of the laws of $\{\mu^{(n)}_t\}_{n\geq1}$ to the law of $\mu_t,$ for fixed $t,$ is established in terms of an appropriate Wasserstein distance.

Our approach is based on ideas from \cite{4}. Recall a definition of the Wasserstein distance between two probability measures. Let $X$ be a separable complete metric space with metric $d$ and the corresponding Borel $\sigma-$field. The set $\cM_p(X)$ of all probability measures $\mu$ on $X$ such that for some (and therefore for an arbitrary) point
$u$ $\int_Xd(u, v)^p\mu(dv)<+ \infty$
is a separable metric space \cite[Theorem 6.18]{5} w.r.t. the distance
$$
W_p(\mu_1, \mu_2)=\left(\inf_{\vk\in\Pi(\mu_1, \mu_2)}\int_{X^2}d(u,v)^p\vk(du, dv)\right)^{1/p} \!\!\!\!, \quad p\geq 1,
$$
where $\Pi(\mu_1, \mu_2)$ is the set of all probability measures on $X^2$ having marginals $\mu_1$ and $\mu_2.$

The measures $\mu_t, \mu^{(n)}_t, n\in\mbN,$  are random elements in $\cM_p(\mbR)$ for any $p\geq 1.$ 
Let $L_t$ and $L^{(n)}_t$ be the laws of $\mu_t$ and $\mu^{(n)}_t$ in $\cM_1(\cM_p(\mbR)),$ respectively. For fixed p, the corresponding Wasserstein distance between probability measures $L^\prime, L^{\prime\prime}$ $\in\cM_1(\cM_p(\mbR))$ is defined via
$$
W_{1}(L^\prime, L^{\prime\prime})=\inf\E \ \! W_p(\mu^\prime, \mu^{\prime\prime}),
$$
where the infinum is taken over the set of pairs of $\cM_p(\mbR)-$valued random elements $\mu^\prime, \mu^{\prime\prime}$ satisfying $\Law(\mu^\prime) = L^\prime, \Law(\mu^{\prime\prime})=L^{\prime\prime}.$ To indicate a specific value of $p$ being used, we write $W_{1,p}$ for the distance on $\cM_1(\cM_p(\mbR)).$ The main result of the first section is the following theorem (cf. \cite{vovch.convergence}[Theorem 1], \cite{4}[Theorem 1.3]). 
\begin{customtheorem}{2.1}
\label{thm1}
Assume that the sequence $\left\{n\delta_n \right\}_{n\in\mbN}$ is bounded. Then for every $p\geq 2$ there
exist a  positive constant $C$ and a number $N\in\mbN$ such that for all $n\geq N$
$$
W_{1,p}(L_t, L^{(n)}_t)\leq C(\log \delta^{-1}_n)^{-1/p}.
$$
\end{customtheorem}

The second part of the paper is devoted to the counting measure associated with the Arratia flow. We discuss the speed of convergence of such measures when one approximates the segment of the real line by its finite subsets. For that, we introduce the multidimensional densities which correspond to different sequences of collisions in the $n-$point motion of the Arratia flow. 

Given an Arratia flow $\{X(u,t)\mid t\ge 0, u\in[0;1]\}$ with zero drift put
 $\Delta_n=\{u_1<\ldots<u_n\}, n\in\mbN,$ and $X_t=\{X(u,t)\mid u\in [0;1]\}.$ The next definition is taken from \cite[Appendix B]{15} (see also \cite{16, 22}) and is adjusted to reflect that the Arratia flow now starts from $[0;1]$ instead of the whole real line.
\begin{defn}
The $n$-point density $p^n_t$  is a measurable function  such that  for any bounded nonnegative measurable $f\colon \mbR^n\to\mbR$
 \begin{equation}
 \label{eq2.0}
 \int_{\mbR^n}f(x)p^n_t(x)dx=\E\sum_
{\begin{subarray}{c}
u_1, \ldots, u_n\in X_t,\\
\mbox{\small all distinct}
\end{subarray}}
f(u_1, \ldots, u_n).
 \end{equation}
\end{defn}
Recall that given $u=(u_1, \ldots, u_n)$ the processes $X(u_1),\ldots, X(u_n)$ are coalescing Brownian motions. To describe  all possible sequences of collisions in this system,   the following notation is used. Define $\mathcal{X}^n\in \left(C([0;1])\right)^n$ by setting $\mathcal{X}^n_j(\cdot) = X(u_j, \cdot), j=\ov{1,n}.$ Let $k$ be the number of distinct values in the set $\{X(u_1, t), \ldots, X(u_n, t)\}.$ Supposing $k<n$ let $\tau_1$ be the moment of the first collision on $[0;t].$ Put 
$j_1 = \min\{i\mid \exists j\neq i \ \mathcal{X}^n_i(\tau_1) = \mathcal{X}^n_j(\tau_1)\}.$ Define $\mathcal{X}^{n-1}\in \left(C([0;1])\right)^{n-1}$ by excluding the $j_1-$th coordinate from $\mathcal{X}^n$. If there exists a moment $\tau_2\le t$ such that for some $i,j\in \{1,\ldots, n-1\}$ $\ \mathcal{X}^{n-1}_i(\tau_2)=\mathcal{X}^{n-1}_j(\tau_2)$ put $j_2$ to be equal to the smallest such number. Repeating the procedure $n-k$ times one obtains a random collection $J_t(u) = (j_1,\ldots, j_{n-k}),$ $j_i \in \{1, \ldots, n-i\}, i= \ov{1, n-k}.$ In the case $k =n$ we set $J_t(u) = \emptyset$ by definition. The set of all possible such collections consisting of $l$ numbers is denoted by $\mathcal{J}_{n,l}.$
\begin{defn}
The random collection $J_t(u)$ defined via the recursive procedure described above is called the coalescence scheme corresponding to the start points $u_1, \ldots, u_n.$ 
\end{defn}
\begin{defn}
\label{def:coal.density}
Given $x=(x_1,\dots,x_n)\in\Delta_n$ the $k$-point density $p^{J, k}_t(x;\cdot)$ corresponding to the coalescence scheme %$J=(j_1,\ldots, j_{n-l})\in\mathcal{J}_{n,n-l}, k\le l,$ 
$J\in\mathcal{J}_{n,n-l}, k\le l,$ and the start points $x_1,\ldots, x_n$ is a measurable function such that  for any bounded nonnegative measurable $f\colon \mbR^k\to\mbR $
 \begin{equation}
 \label{eq2.1}
 \int_{\mbR^k}p^{J, k}_t(x;y)f(y)dy=
\E\sum_{{\begin{subarray}{c}
u_1, \ldots, u_k\in \{X(x_1,t),\ldots, X(x_n,t)\},\\
\mbox{\small all distinct}
\end{subarray}}}
f(u_1, \ldots, u_k)  \times \1\left(J_t(x) = J\right).
 \end{equation}
\end{defn} 
The integral representation is obtained for such densities (Theorem \ref{thm2.1}). The result on convergence of the multidimensional densities  given  in Theorem \ref{thm2.2} is motivated by the discrete approximations of Section 1.

Consider the vectors $U^{(n)}=(u^{(n)}_1, \ldots, u^{(n)}_n) \in \Delta^n,$ such that  $u^{(n)}_1=0, u^{(n)}_n=1, n\in\mbN,$ 
$$
\limsup_{n\to\infty}\max_{j=\ov{0, n-1}}\left(u^{(n)}_{j+1}-u^{(n)}_j\right)=0,
$$
and
$$
\left\{u^{(n)}_1, \ldots, u^{(n)}_n\right\}\subset \left\{u^{(n+1)}_1, \ldots, u^{(n+1)}_{n+1}\right\}, \quad n\in\mbN.
$$
Define
\begin{equation}
\label{eq:def.p.k.t.}
p^k_t(U^{(n)};\cdot)=\sum_{i=k}^{n}\sum_{J\in\mathcal{J}_{n,n-i}}
p^{J, k}_t(U^{(n)};\cdot), \quad k =\ov{1,n}, n\in\mbN.
\end{equation}

\begin{customtheorem}{3.3}
\label{thm2.3}
There exists an absolute positive constant $C$ such that
$$
0\leq {p}^1_t(y)-p^1_t(U^{(n)};y)\leq C \max_{j=\ov{1, n-1}}\left(u^{(n)}_{j+1}-u^{(n)}_j\right)^2
$$
for almost all $y.$
\end{customtheorem}

\section{The Wasserstein distance between $L_t$ and $L^{(n)}_t$}
We approximate the measures $\mu_t$ and $\mu^{(n)}_t$ with point measures 
\begin{align*}
\mu^{(n),m}_t & =m^{-1}\sum^{m-1}_{j=0}\delta_{\Phi^{(n)}_t(j/m)},\\
\mu^{m}_t & =m^{-1}\sum^{m-1}_{j=0}\delta_{\Phi^{a}_t(j/m)}, \quad n,m\in\mbN.
\end{align*}

We begin with $L_p-$ estimates on the divergence between two solutions of a one-dimensional SDE in terms of the difference of the initial points and estimates of the same type for their approximations via the fractional step method.

Let $a$ be a bounded function satisfying the Lipschitz condition with constant $C_a.$  Put $M_a=\sup_\mbR|a|.$ Given a standard Brownian motion $w$ and a point $u\in\mbR$ the equation
\begin{equation*}
\label{eq1}
\begin{split}
dx(t)&=a(x(t))dt+dw(t),\\
x(0)&=u, \quad t\in[0;1],
\end{split}
\end{equation*}
has the unique strong solution $x.$ Consider, for $t\in \Delta^{(n)}_j, j=\ov{0, n-1}: $
\begin{equation}
\label{eq1.1}
\begin{split}
y^{(n)}(t)&=u+\int^{t^{(n)}_{j+1}}_0 a(z^{(n)}(s))ds+w(t),\\
z^{(n)}(t)&=u+\int^{t}_0a(z^{(n)}(s))ds+w(t^{(n)}_j).
\end{split}
\end{equation}
We will encode such a relation between $x, y^{(n)}, z^{(n)}$ and $w, u$ by writing $x=D(w,u),$  $(y^{(n)}, z^{(n)})=S^{(n)}(w,u).$
The next result is a straightforward generalization of  \cite[Corollary 4.2]{11}.
\begin{lem}
\label{lem1}
For any $p\geq1$ there exists $C > 0$ such that
\begin{align*}
\E\sup_{s\leq1}\left|x(s)-y^{(n)}(s)\right|^p &\leq C \delta^{p/2}_n,\\
\sup_{s\leq1} \E\left|x(s)-z^{(n)}(s)\right|^p &\leq C \delta^{p/2}_n.
\end{align*}
\end{lem}
%%%%%%%%%%%%%%%%%%%%%%
\begin{lem}
\label{lem2}
Suppose $u_1, u_2\in\mbR,$ and $w_1, w_2$ are independent Brownian motions. Let $x_k=D(w_k, u_k), k=1,2.$ Then for any $p\geq 1$ 
there exists $C > 0$ such that
\begin{align*}
\E\sup_{s\leq1}\left|x_1(s\wedge\theta)-x_2(s\wedge\theta)\right|^p & \leq C\Big(|u_1-u_2| + |u_1-u_2|^p\Big), \quad p\ge 2, \\
\E\sup_{s\leq1}|x_1(s\wedge\theta)-x_2(s\wedge\theta)|^p & \leq C\left(|u_1-u_2|^{p/2} + |u_1-u_2|^{p}\right), \quad p\in[1;2),
\end{align*}
where $\theta=\inf\{1;s\mid x_1(s)=x_2(s)\}.$
\end{lem}
\begin{proof}
Denote $\Delta u=u_2-u_1, \Delta x=x_2-x_1.$ Assume $u_2>u_1.$ Consider the SDE
\begin{align*}
d\eta(t)&=C_a\eta(t)dt+dw_2(t)-dw_1(t),\\
\eta(0)&=\Delta u,
\end{align*}
with the unique strong solution 
\begin{equation}
\label{eq2}
\eta(t)=\e^{C_at}\Delta u+\sqrt{2}\e^{C_at}\int^t_0\e^{-C_as}dw(s),
\end{equation}
where $w=\frac{w_2-w_1}{\sqrt{2}}.$ %The process
We have 
$$
\eta(t)-\Delta x(t)=C_a\int_0^t\left(\eta(s)-\Delta x(s)\right)ds+\int^t_0\left(C_a\Delta x(s)-a(x_2(s))+a(x_1(s))\right)ds,
$$
therefore a.s.
\begin{equation}
\label{eq.l2.1}
\eta(t)-\Delta x(t)=\e^{C_at}\int^t_0\e^{-C_a s}\left(C_a\Delta x(s)-a(x_2(s))+a(x_1(s))\right)ds \ge 0, \quad t\in[0; \theta].
\end{equation} 
Applying the Knight theorem \cite[Prop.18.8]{12} to the stochastic integral in \eqref{eq2}, we get
$$
\eta(t)=\e^{C_at}\Delta u+\sqrt{2}\e^{C_at}\beta\left(\int^t_0\e^{-2C_as}ds\right),
$$
where $\beta$ is some Brownian motion. Then \eqref{eq.l2.1} implies
$$
\theta\leq\vk=\inf\left\{1; s\mid\eta(s)=0\right\}=\inf\left\{1;s\mid\beta\left(\int^t_0\e^{-2C_as}ds\right)=\frac{-\Delta u}{\sqrt{2}}\right\}.
$$
Thus
$$
\E\sup_{t\leq\theta}|\Delta x(t)|^p\leq \E\sup_{t\leq\vk}|\eta(t)|^p\leq 2^{p-1}\e^{pC_a}\Delta u^p+2^{3p/2-1}\e^{pC_a}\E\sup_{t\leq\vk}|\beta(t)|^p,
$$
since $\int^t_0\e^{-2C_as}ds < t, t\ge 0.$ The same reason implies that the random moment $\vk$ is a stopping time w.r.t the filtration generated by $\{\beta(t) \mid t\in [0;1]\},$ therefore, by the Burkholder-Davis-Gundy inequality,  
$$
\E\sup_{t\leq\vk}|\beta(t)|^p\leq C_p\E\vk^{p/2}, \quad p\ge 2,
$$
for positive constants $C_p.$
The distribution of $\vk$ is given via
$$
\Prob\left(\vk\geq t\right)=\sqrt{\frac{2}{\pi}}\int^{a(t)}_0\e^{-y^2 /2}dy, \ a(t)=\frac{C_a^{1/2}(u_2-u_1)}
{\left(1-\e^{-2t}\right)^{1/2}},
$$
hence for fixed $p\ge 2$
\begin{align}
\label{eq3}
\E\vk^{p/2} &=\frac{p}{2}\int^1_0t^{\frac{p}{2}-1}\left(\sqrt{\frac{2}{\pi}}\int^{a(t)}_0\e^{-y^{2}/2} dy\right)dt\leq \frac{p}{\sqrt{2\pi}} \int^1_0a(t)t^{p/2-1}dt\leq
C (u_2-u_1)
\end{align}
for some $C.$ To handle the case $p \in[1;2)$ one uses the Lyapunov inequality and the foregoing estimates.
\end{proof}

We consider a modification of \eqref{eq1.1}: on every $\Delta_j^{(n)}, j=\ov{0,n-1},$
\begin{align*}
y^{(n)}(t) &=u_y+\int^{t^{(n)}_{j+1}}_0a(z^{(n)}(s))ds+w(t), \\
z^{(n)}(t) &=u_z+\int_0^t a(z^{(n)}(s))ds+w(t^{(n)}_j), \quad t\in\Delta_j^{(n)},
\end{align*}
where nonrandom $u_y$ and $u_z$ are not necessarily equal. The pair $(y^{(n)},z^{(n)} )$ is denoted by $S^{(n)}(w, u_y, u_z).$

\begin{lem}
\label{lem3} 
Assume that the sequence $\left\{n\delta_n \right\}_{n\in\mbN}$ is bounded. Let $u_{y_1}, u_{y_2}, u_{z_1}, u_{z_2}\in\mbR,$ and let $w_1, w_2$ be  independent standard Brownian motions. Put $(y^{(n)}_k, z^{(n)}_k)=S^{(n)}(w_k, u_{y_k}, u_{z_k}), k=1,2.$ Then for any $p\geq 2$ and for any $\ve\in(0; \frac{1}{2})$  there exist $C>0$ and  $N\in\mbN$ such that for all $ n\geq N$ 
\begin{align*}
\E \sup_{s\leq 1}\left|y^{(n)}_1(s\wedge\theta^{(n)})-y^{(n)}_2(s\wedge\theta^{(n)})\right|^p &\leq C \left(\delta^{1/2-\ve}_n +  \sum_{l=1}^2 \left(|u_{z_l}-u_{y_l}| + |u_{z_l}-u_{y_l}|^p \right) + \right. \\ 
& \phantom{aaaaaa} + |u_{y_2}-u_{y_1}|^p+|u_{y_2}-u_{y_1}|\Bigg),
\end{align*}
where  $\theta^{(n)}=\inf\{1; s \mid y^{(n)}_2(s)=y^{(n)}_1(s)\}.$
\end{lem}
\begin{proof}
We extend the proof of Lemma \ref{lem2}. Suppose  $u_{z_2}-u_{z_1}\geq 0, $ $u_{y_2}-u_{y_1}\geq 0.$ Denote $\Delta u_y=u_{y_2}-u_{y_1}, \Delta u_z=u_{z_2}-u_{z_1},$ and let $\eta$ be defined as in \eqref{eq2} with $\Delta u=\Delta u_y.$ Then for $t\leq \theta^{(n)}, t\in \Delta^{(n)}_j$ for some $j,$ and for $\Delta y=y_2-y_1$
\begin{align*}
\Delta y(t)-\eta(t)&=C_a\int^t_0(\Delta y(s)-\eta(s))ds+
\int^t_0\left(a(z^{(n)}_2(s))-a(z^{(n)}_1(s))-C_a\Delta y(s)\right)ds+ \\
& \phantom{aaaaaa} +\int^{t^{(n)}_{j+1}}_t(a(z^{(n)}_2(s))-a(z^{(n)}_1(s)))ds\leq  \\ 
& \leq C_a\int^t_0(\Delta y(s)-\eta(s))ds+ C_a\int^t_0\sum^2_{l=1} (-1)^{l}
(z^{(n)}_l(s)-y^{(n)}_l(s))ds+2\delta_n M_a, 
\end{align*}
since $z^{(n)}_2\geq z^{(n)}_1$ on $[0; \theta^{(n)}].$ For $s\in \Delta^{(n)}_i, i\leq j,$
\begin{align*}
\left| z^{(n)}_k(s)-y^{(n)}_k(s) - w_k(t^{(n)}_i)+w_k(s)\right|& \leq
\int^{t^{(n)}_{i+1}}_s|a(z^{(n)}_k(s)|ds  + |u_{z_k}-u_{y_k}|\leq \\ 
&\leq (t^{(n)}_{i+1}-s)M_a+|u_{z_k}-u_{y_k}|,
\end{align*}
so it follows, for $t\in \Delta^{(n)}_j, t\leq\theta^{(n)},$ that
\begin{align*}
\Delta y(t)-\eta(t) & \leq C_a\int^t_0(\Delta y(s)-\eta(s))ds + C_a\sum^{j-1}_{k=0}\int_{\Delta^{(n)}_k}\sum^2_{l=1}(-1)^{l}(w_l(t^{(n)}_{k+1})-w_l(s))ds  \\ 
& \phantom{aaaaaa} +2C_aM_a\sum^{j-1}_{k=0}\int_{\Delta^{(n)}_k}(t^{(n)}_{k+1}-s)ds+2\delta_nM_a+\sum_{l=1}^2 |u_{z_l}-u_{y_l}|.
\end{align*}
Since
$$
2C_aM_a\sum^{j-1}_{k=0}\int_{\Delta^{(n)}_k}(t^{(n)}_{k+1}-s)ds\leq C_aM_a\delta_n,
$$
the Gronwall--Bellman inequality implies 
\begin{align*}
\Delta  y(t)\leq \eta(t)+\e^{C_a}M_a(C_a+2)\delta_n+\e^{C_a}\sum_{l=1}^2 |u_{z_l}-u_{y_l}|+ \e^{C_a}C_a \max_{j=\ov{1,n}}\left|\xi_j\right|, 
\end{align*}
where 
$$
\xi_j=\sum^2_{l=1}\sum^{j-1}_{k=0}\int_{\Delta^{(n)}_k}(-1)^l(w_l(t^{(n)}_{k+1})-w_l(s))ds, j=\ov{1,n}.
$$
Thus 
\begin{align}
\label{eq4}
\E\sup_{s\leq\theta^{(n)}}\left|\Delta y(s)\right|^p& \leq 4^{p-1}\left( \E\sup_{s\leq\theta^{(n)}}\left|\eta(s)\right|^p + \e^{pC_a}M^p_a(C_a+2)^p\delta^p_n +\e^{pC_a}\left(\sum_{l=1}^2 |u_{z_l}-u_{y_l}|\right)^p + \right. \nonumber \\ 
&\phantom{aaaaaaa} + \left. \e^{pC_a}C_a^p\ \E\max_{j=\ov{1,n}}\left|\xi_j\right|^p \right).
\end{align}
The random variables $\xi_{j+1} -\xi_j, j =\ov{1,n-1},$ are independent centered Gaussian variables; $\Var(\xi_n) \leq 2n\delta_n^2.$ Therefore, by the Levy inequality, there exists a constant $C$ such that 
\begin{equation}
\label{eq:levy.1}
\E\max_{j=\ov{1,n}}\left|\xi_j\right|^p \leq 2 \E |\xi_n|^p \leq C n^p \delta_n^{2p}
\end{equation}
and, for any $x_n > 0,$
\begin{align}
\label{eq:levy.2}
\Prob\left(\max_{k=\ov{1,n}}|\xi_k|\geq x_n\right) \leq 2\Prob\left( \left|\mathbb{N}(0,1)\right| \geq \frac{x_n}{\left(\Var(\xi_n)\right)^{1/2}} \right) \leq C \frac{n^{1/2}\delta_n}{x_n} \e^{-\frac{x_n^2}{4n\delta_n^2}}.
\end{align}

At the same time, proceeding exactly as in the proof of Lemma \ref{lem2} we obtain
\begin{equation}
\label{eq6}
\E\sup_{t\leq\theta^{(n)}}|\eta(t)|^p\leq 2^{p-1}\e^{pC_a}\Delta u_y^p+2^{3p/2-1}\e^{pC_a}C_p \E(\theta^{(n)})^{p/2}.
\end{equation}
However, at time $\theta^{(n)}$
\begin{align*}
\eta(\theta^{(n)}) &\geq -\e^{C_a}M_a(C_a+2)\delta_n-\e^{C_a}\sum_{l=1}^2 |u_{z_l}-u_{y_l}|-\e^{C_a}C_a\max_{j=\ov{1,n}}|\xi(j)|,
\end{align*}
so for fixed $x_n>0$
\begin{equation*}
%\label{eq7}
\E(\theta^{(n)})^{p/2}\leq \Prob\left(\max_{k=\ov{1,n}}|\xi_k|\geq x_n\right)+\E\tau^{p/2}_n,
\end{equation*}
where 
$$
\tau_n=\inf\left\{1; s\mid\eta(s)=-\e^{C_a}M_a(C_a+2)\delta_n-\e^{C_a}\sum_{l=1}^2 |u_{z_l}-u_{y_l}|-\e^{C_a}C_a x_n\right\}.
$$
Put $K = \sup_{k\in\mbN} k \delta_k.$ Reasoning leading to \eqref{eq3}, when combined with \eqref{eq:levy.2}, implies that 
\begin{align}
\label{eq8}
\E(\theta^{(n)})^{p/2} & \leq C \Big(  \Delta u_y + \delta_n + \sum_{l=1}^2 |u_{z_l}-u_{y_l}| + x_n + K^{1/2}\frac{\delta_n^{1/2}}{x_n}  \e^{-\frac{x_n^2}{4K\delta_n}} \Big),
\end{align}
for the redefined  constant $C.$ Choosing $x_n=\delta_n^{1/2-\ve},$ for any fixed $\ve\in(0;\frac{1}{2}),$ and substituting \eqref{eq:levy.1}, \eqref{eq6} and \eqref{eq8} into \eqref{eq4} finishes the proof. 
\end{proof}

Let us recall the definitions of the measures considered. For the random elements in $\cM_p(\mbR)$
\begin{align*}
\mu_t & =\lambda\circ\left(\Phi^a_t\right)^{-1}, 
\ \mu^m_t=\left(\frac{1}{m}\sum^m_{j=1}\delta_{j/m}\right)\circ\left(\Phi^a_t\right)^{-1}, \\
\mu^{(n)}_t &=\lambda\circ\left(\Phi^{(n)}_t\right)^{-1},
\ \mu^{(n),m}_t=\left(\frac{1}{m}\sum^m_{j=1}\delta_{j/m}\right)\circ\left(\Phi^{(n)}_t\right)^{-1}, \quad n, m\in\mbN,
\end{align*}
we consider their distributions as elements of $\cM_1(\cM_p(\mbR)):$
\begin{align*}
L_t &=\Law(\mu_t), \ L^m_t=\Law(\mu^m_t), \\
L^{(n)}_t&=\Law\left(\mu^{(n)}_t\right), \ L^{(n),m}_t=\Law\left(\mu^{(n),m}_t\right), \quad n,m\in\mbN.
\end{align*}

Analogously to \cite[Theorem 2.1]{4}, we have
\begin{lem}
\label{lem4}
For any $p\geq 2$ there exists $C>0$ such that 
$$
W_{1,p}(L_t, L^m_t)\leq Cm^{-1/p},
$$
and, if additionally $\left\{n \delta_n \right\}_{n\in\mbN}$ is bounded,
$$
W_{1,p}(L^{(n)}_t, L^{(n),m}_t)\leq C\left(m^{-1}+\delta^{1/2-\ve}_n\right)^{1/p}.
$$
\end{lem}
\begin{proof}
Since the random measures $(\mu_t, \mu^m_t)$ is a coupling for the pair $(L_t, L^m_t),$ it follows from the definition of the distance $W_{1,p}$ that 
$$
W_{1,p}(L_t, L^m_t)\leq \E \ \! W_p(\mu_t, \mu^m_t) ,
$$
therefore, by \cite[Theorem 2.18, Remark 2.19]{5} and Lemma \ref{lem2}, for some $C,$
\begin{align*}
W_{1,p}(L_t, L^m_t)& \leq \left(\sum^{m-1}_{j=0}\int^{(j+1)/m}_{j/m}
\E\left|\Phi^a_t(y)-\Phi^a_t(j/m)\right|^pdy
\right)^{1/p}\leq \\
& \leq 2^{1/p} C \left(\sum^{m-1}_{j=0}\int^{(j+1)/m}_{j/m}(y-j/m)dy\right)^{1/p}= Cm^{-1/p}
\end{align*}
as, for $x_1, x_2$ from Lemma \ref{lem2}, 
$$
\left\{\left(\Phi^a_{t\wedge\theta_1}(y), \Phi^a_{t\wedge\theta_1}\left(j/m\right)\right)\mid t\in[0; 1]\right\} \overset{d}{=} \left\{\left(x_1(t\wedge\theta_2), x_2(t\wedge\theta_2\right)\mid t\in[0;1]\right\},$$
$\theta_1, \theta_2$ being he moments of meeting for the corresponding pairs of processes. Similarly, using Lemma \ref{lem3} with $u_{z_k}=u_{y_k}, k=1,2,$
$$
W_{1,p}\left(L^{(n)}_t, L^{(n),m}_t\right) \leq
C \left(\sum^{m-1}_{j=0}\int^{(j+1)/m}_{j/m}\left(y-j/m\right)dy+\delta_n^{1/2-\ve}\right)^{1/p} \leq
C \left(m^{-1}+\delta^{1/2-\ve}_n\right)^{1/p},
$$
for some $C.$
\end{proof}

Now we describe  appropriate couplings for $(\mu^m_t, \mu^{(n), m}), n\in\mbN,$ given fixed $m.$ Suppose $w_1, \ldots, w_m$ are independent standard Brownian motions. Denoting $u_j=j/m, j=\ov{0,m},$ put
\begin{align*}
x_j&=D(w_j, u_j), \\
(y^{(n)}_j, z^{(n)}_j)&=S^{(n)}(w_j, u_j), \quad n\in\mbN,
\end{align*}
and define $\wt{x}_1=x_1, \wt{y}^{(n)}_1=y^{(n)}_1, \wt{z}^{(n)}_1=z^{(n)}_1.$ Proceeding recursively, put
\begin{align*}
\theta_j&=\inf\{1; s\mid x_j(s)=\wt{x}_{j-1}(s)\}, \\
\theta^{(n)}_j&=\inf\{1; s\mid y_j^{(n)}(s)=\wt{y}^{(n)}_{j-1}(s)\},\\
\wt{x}_j(t)&=x_j(t)\1\left(t<\theta_j\right)+\wt{x}_{j-1}(t)\1\left(t\geq\theta_j\right), \\
\wt{y}^{(n)}_j(t)&=y^{(n)}_j(t)\1\left(t<\theta^{(n)}_j\right)+\wt{y}^{(n)}_{j-1}(t)\1\left(t\geq\theta^{(n)}_j\right), \quad j=\ov{2, m}.
\end{align*}
Consider a random  number $k^{(n)}_j$ such that $\theta^{(n)}_j\in \Delta^{(n)}_j$ and put
\begin{align*}
\wt{z}^{(n)}_j(t) &=z^{(n)}_j(t)\1\left(t<t^{(n)}_{k^{(n)}_j+1}\right)+\wt{z}^{(n)}_{j-1}(t)\1\left
(t\geq t^{(n)}_{k^{(n)}_j+1}\right), \quad t\in[0;1),\ j=\ov{2,m}.
\end{align*}
Values at $t=1$ are taken to be equal to the corresponding left limits. The processes 
\begin{align*}
\wt{w}_1&=w_1, \wt{w}^{(n)}_1=w_1, \\
\wt{w}_j(t) & =w_j(t)\1\left(t<\theta_j\right)+\wt{w}_{j-1}(t)\1\left(t\geq\theta_j\right), \\
\wt{w}^{(n)}_j(t)&=w_j(t)\1\left(t<\theta^{(n)}_j\right)+\wt{w}^{(n)}_{j-1}(t)\1\left(t\geq\theta^{(n)}_j\right), \quad j=\ov{2,m}, \ n\in\mbN,
\end{align*}
can be checked to be Brownian motions.

The proofs of the next two lemmas are based on the repeated application of \eqref{eq1.1} and are thus omitted.
\begin{lem}
\label{lem5}
For $n, m\in\mbN$  and $j=\ov{1,m},$
\begin{align*}
\wt{x}_j &=D(\wt{w}_j, u_j), \\
\left(\wt{y}^{(n)}_j, \wt{z}^{(n)}_j\right) &=S^{(n)}(\wt{w}^{(n)}_j, u_j).
\end{align*}
\end{lem}

\begin{lem}
\label{lem6}
For $n, m\in \mbN$ 
\begin{align*}
\left(\Phi^a(u_1), \ldots, \Phi^a(u_m)\right)&\overset{d}{=}(\wt{x}_1, \ldots, \wt{x}_m), \\
\left(\Phi^{(n)}_{0, \cdot}(u_1), \ldots, \Phi^{(n)}_{0,\cdot}(u_m)\right)&\overset{d}{=}(\wt{y}^{(n)}_1, \ldots, \wt{y}^{(n)}_m),
\end{align*}
in $(D([0; 1]))^m.$
\end{lem}
\begin{proof}[Proof of Theorem \ref{thm1}] Repeating the reasoning of the proof of Lemma \ref{lem4} and using Lemma \ref{lem6}, we get that
$$
\left(W_{1,p}\left(L^m_t, L^{(n),m}_t\right)\right)^p\leq \sum^{m-1}_{j=0}\int^{(j+1)/m}_{j/m}
\E\left|\wt{x}_j(t)-\wt{y}^{(n)}_j(t)\right|^pdu=m^{-1}
\sum^{m-1}_{j=0} \E\left|\wt{x}_j(t)-\wt{y}^{(n)}_j(t)\right|^p.
$$
By Lemma \ref{lem1}, for some positive $C_1$
$$
\E\left|\wt{x}_1(t)-\wt{y}^{(n)}_1(t)\right|^p\leq \E\sup_{s\leq 1}\left|x_1(s)-y^{(n)}_1(s)\right|^p\leq C_1\delta^{p/2}_n.
$$
Continuing for $j=2,$ 
\begin{align*}
\E \left|\wt{x}_2(t)-\wt{y}^{(n)}_2(t)\right|^p & = \E\left|\wt{x}_2(t)-\wt{y}^{(n)}_2(t)\right|^p\times
\left[\1\left(t\geq\theta_1\wedge\theta^{(n)}_1\right)+\1\left(\theta^{(n)}_1\leq t<\theta_1\right)+ \right.\\
& \phantom{aaaaaa} \left. +\1\left(\theta_1\leq t<\theta_1^{(n)}\right)+\1\left(t<\theta_1^{(n)}\wedge\theta_1\right)\right]\leq \\
& \leq
\E\left|\wt{x}_1(t)-\wt{y}^{(n)}_1(t)\right|^p \1\left(t \geq\theta_1\wedge\theta_1^{(n)}\right)+ \\
& \phantom{aaaaaa} +\E\left|{x}_2(t)-{y}^{(n)}_2(t)\right|^p \1\left(t<\theta_1^{(n)}\wedge\theta_1\right)+  \\
& \phantom{aaaaaa} +
2^{p-1}\E\left[\left|{x}_2(t)-\wt{x}_1(t)\right|^p+\left|\wt{x}_1(t)-\wt{y}^{(n)}_1(t)\right|^p\right] \1\left(\theta_1^{(n)}\leq t<\theta_1\right)+ \\
& \phantom{aaaaaa} +
2^{p-1}\E\left[\left|\wt{x}_1(t)-\wt{y}^{(n)}_1(t)\right|^p+\left|\wt{y}^{(n)}_1(t)-{y}^{(n)}_2(t)\right|^p\right] \1\left(\theta_1\leq t<\theta_1^{(n)}\right)\leq \\
&\leq
2^{p-1}\E\left|\wt{x}_1(t)-\wt{y}^{(n)}_1(t)\right|^p+2^{p-1}\E\left|\wt{x}_1(t)-{x}_2(t)\right|^p
 \1\left(\theta_1^{(n)}\leq t<\theta_1\right)+ \\ 
& \phantom{aaaaaa} +
2^{p-1}\E\left|\wt{y}^{(n)}_1(t)-{y}^{(n)}_2(t)\right|^p
\1\left(\theta_1\leq t<\theta_1^{(n)}\right)+\E\left|x_2(t)-y^{(n)}_2(t)\right|^p.
\end{align*}
Using Lemma \ref{lem1} again we obtain
\begin{align}
\label{eq9}
\E\left|\wt{x}_2(t)-\wt{y}^{(n)}_2(t)\right|^p &\leq(2^{p-1}+1)C_1\delta^{p/2}_n+ 2^{p-1} \E\sup_{\theta_1^{(n)}\leq s\leq\theta_1}
\left|\wt{x}_1(s)-x_2(s)\right|^p \1\left(\theta_1^{(n)}\leq\theta_1\right)+ \nonumber \\ 
& \phantom{aaaaaa} +2^{p-1}\E\sup_{\theta_1\leq s\leq\theta_1^{(n)}}
\left|\wt{y}^{(n)}_1(s)-y^{(n)}_2(s)\right|^p \1\left(\theta_1\leq\theta_1^{(n)}\right).
\end{align}
Consider the last two summands in \eqref{eq9} separately. Note that $\wt{x}_1$ and $x_2$ are independent and such are $\wt{y}_1$ and $y_2,$ whence one can deduce, using the Markov property, that 
\begin{align}
\label{eq10}
\E\sup_{\theta_1^{(n)}\leq s\leq\theta_1}\left|\wt{x}_1(s)-x_2(s)\right|^p \1\left(\theta_1^{(n)}\leq\theta_1\right)\leq \E\sup_{0\leq s\leq\tau_1}\left| \eta_1(s)-\eta_2(s)\right|^p,
\end{align}
where $\eta_k=D(\beta_k, v_k), k=1,2,$ with  $\beta_1, \beta_2$ being independent Brownian motions, also independent of $w_1, w_2$ (and therefore of $\wt{x}_1, x_2$), and
\begin{align*}
v_1 &=\wt{x}_1(\theta_1^{(n)}), \ \ v_2=x_2(\theta_1^{(n)}), \\
\tau_1 & =\inf\left\{1; s\mid\eta_1(s)=\eta_2(s)\right\}.
\end{align*}
Thus, by the first inequality of Lemma \ref{lem1}, for any $q\geq 1,$
$$
\E\left|v_1-v_2\right|^q\leq \E\left|\wt{x}_1(\theta_1^{(n)})-\wt{y}^{(n)}_1(\theta_1^{(n)})\right|^q+ \E\left|x_2(\theta_1^{(n)})-y^{(n)}_2(\theta_1^{(n)})\right|^q\leq 2C_1\delta^{q/2}_n,
$$
so after taking the conditional expectation in \eqref{eq10} and averaging over $v_1, v_2$ one gets due to Lemma \ref{lem2}
\begin{equation}
\label{eq11}
\E\sup_{\theta_1^{(n)}\leq s\leq\theta_1}\left|\wt{x}_1(s)-x_2(s)\right|^p \1\left(\theta_1^{(n)}\leq\theta_1\right)\leq  C_2\delta^{1/2}_n,
\end{equation}
for some $C_2.$ Similarly,
\begin{align*}
\E\sup_{\theta_1\leq s\leq\theta_1^{(n)}}\left|\wt{y}^{(n)}_1(s)-y_2^{(n)}(s)\right|^p \1\left(\theta_1\leq\theta_1^{(n)} \right)\leq
\E\sup_{0\leq s\leq\tau_2}\left|\xi_1(s)-\xi_2(s)\right|^p,
\end{align*}
where $\xi_k=\left(S^{(n)}(\beta_k, v_{k1}, v_{k2} )\right)_1, k=1,2,$ and 
\begin{align*}
v_{11}&=\wt{y}^{(n)}_1(\theta_1), \ v_{12}=\wt{z}^{(n)}_1(\theta_1), \ v_{21}=y^{(n)}_2(\theta_1), \ v_{22}=z^{(n)}_2(\theta_1), \\
\tau_2 & =\inf\{1; s\mid\xi_1(s)=\xi_2(s)\}.
\end{align*}
Using both inequalities of Lemma \ref{lem1}, 
applying Lemma \ref{lem3} with $u_{y_1}=v_{11}, u_{y_2}=v_{21}, u_{z_1}=v_{12}, u_{z_2}=v_{22}$ and taking expectation one can show that for some positive $C_3$
\begin{equation}
\label{eq12}
\E\sup_{\theta_1\leq s\leq\theta_1^{(n)}}\left|\wt{y}^{(n)}_1(s)-y_2(s)\right|^p \1\left(\theta_1\leq\theta_1^{(n)} \right)\leq C_3\delta^{1/2-\ve}_n.
\end{equation}
Substituting \eqref{eq11} and \eqref{eq12} into \eqref{eq10} gives, for some $C_4>1,$
\begin{align*}
\E &\left|\wt{x}_2(t)-\wt{y}^{(n)}_2(t)\right|^p\leq C_4\delta^{1/2-\ve}_n,
\end{align*}
starting from some $N$ independent of $m.$ 
Using such an estimate recursively for $j=3, \ldots, m$ one finally concludes that 
$$
\sum^{m-1}_{j=0}\E\left|\wt{x}_j(t)-\wt{y}^{(n)}_j(t)\right|^p\leq\sum^{m-1}_{j=1}C^j_4\delta^{1/2-\ve}_n. 
$$
By Lemma \ref{lem4} there exist positive $C_5$ and a number $N^\prime \geq N$ such that for any $n\geq N^\prime$
$$
W_{1,p}\left(L_t, L^{(n)}_t\right)\leq C_5\left(m^{-1}+\delta^{1/2-\ve}_n\right)^{1/p}+C_5\left(C^m_4\delta^{1/2-\ve}_n\right)^{1/p},
$$
therefore choosing $m=m(n)$ in such a way that $m(n)=(\frac{1}{4} -\frac{\ve}{2})\frac{\log\delta^{-1}_n}{\log C_4}$ concludes the proof.
\end{proof}

\section{On counting measures associated with the Arratia flow}
Recall  that $\Delta_n=\{u_1<\ldots<u_n\}, n\in\mbN,$ and $\{X(u,t)\mid t\ge 0, u\in[0;1]\}$ is an Arratia flow with zero drift. 
Denote the density of a standard $m$-dimensional Brownian motion killed upon exiting $\Delta_m$ by $p^m_{0,t}.$ This density is given via the Karlin-McGregor determinant
$$
p^m_{0,t}(x; y)=\det\|g_t(x_i-y_j)\|_{i,j=\ov{1,m}}, \quad x,y\in\Delta_m,
$$
where
$g_t(a)=\frac{1}{\sqrt{2\pi t}}e^{-a^2/2t}.$
 
Any $J=(j_1,\ldots, j_{n-k})\in\mathcal{J}_{n,n-k}$ can be associated with a partition of the set $\{1,\ldots, n\}$ by the following procedure. Starting from the partition consisting of singletons, at each step $i =1,\ldots, n - k$ proceed by merging two subsequent blocks in the current partition with the numbers $j_i$ and $j_i+1,$ the blocks being listed in order of appearance w.r.t. the usual ordering of $\mbN.$ The resulting partition will be denoted by $\pi(J);$ the blocks of $\pi(J),$ by $\pi_1(J), \ldots, \pi_k(J).$ Note that
$$
\left\{ J_t(u) = J \right\} =\left\{ \forall j\in \pi_i(J) \ X(x_j, t)=X(x_{\min \pi_i(J)}, t), i=\ov{1, k} \right\}.   
$$
\begin{lem}
  \label{lem2.1}
For all  $t\in[0; 1],$ $x\in\Delta_n,$ $k\in\{1,\ldots, n\}$ and $J=(j_1,\ldots, j_{n-m}) \in\mathcal{J}_{n,n-m}, m\ge k,$  the density $p^{J, k}_t(x;\cdot)$ exists. Moreover,  $p^{J, k}_t(x;\cdot)\leq p^k_{0,t}(x; \cdot)$ a.e. if $m=k.$
  \end{lem}
  \begin{proof} 
Suppose $k=m.$ Let $A$ be a Borel subset of $\Delta_k.$ Define a mapping $T\colon \Delta_n \mapsto \Delta_k$ by the rule $T(u)_l =u_{\min \pi_l(J)},$ $l =\ov{1,k}.$ Then 
\begin{align*} 
  \E&\sum_{u_1, \ldots, u_k\in \{X(x_1,t),\ldots, X(x_n,t)\}}
\1_{A}\left(u_1, \ldots, u_k\right)  \times
 \1\left(J_t(x) = J\right)\leq \\
  & \phantom{aaaaaa}\leq \E\1\left( T\left(X(x_1, t), \ldots, X(x_n, t))\right)\in A\right)= \\
  &\phantom{aaaaaa}=\int_{A}p^k_{0,t}(x; y)dy.
\end{align*}
 The Radon-Nikodym theorem yields the claim of the lemma. The cases when $A$ is not a subset of $\Delta_k$ and $m\not= k$ are treated similarly.
 \end{proof}

It  is possible to derive an explicit expression for $p^{J, k}_t.$
Denote the boundary of $\Delta_n$ by $\pt\Delta_n.$ Additionally, define
$$
\pt\Delta_{n,j}=\{(u_1, \ldots, u_n)|u_1<\ldots<u_j=u_{j+1}<\ldots<u_n\}, \quad j=\ov{1,n-1}.
$$
Let $w=(w_1, \ldots, w_n)$ be a standard Brownian motion.  We write $\E_{r,z}$ for the mathematical expectation calculated w.r.t. the distribution of $(w_1, \ldots, w_n)$ started at $r$ from $z.$ Define $\Delta_n(a)=\{u_1<\ldots<u_n \leq a\}, n\in\mbN.$
\begin{thm}
\label{thm2.1}
For all $t\in[0; 1]$ and $J=(j_{1}, \ldots, j_{n-k})\in\mathcal{J}_{n,n-k}$ and $x\in\Delta_n$ a.e.
\begin{align*}
p_t^{J, k}(x;y) & =\int_{\Delta_{n-k}(t)}dt_1\ldots dt_{n-k}
\int_{\pt\Delta_{n,j_{1}}}m(dz_1)
\int_{\pt\Delta_{n,j_{2}}}m(dz_2)\ldots
 \int_{\pt\Delta_{k+1,j_{n-k}}} \!\! m(dz_{n-k}) (-1)^k2^{-k}\times \\
  & \phantom{aaaaaa}  \times \frac{\pt}{\pt\nu_{z_1}}p^n_{0,t_1}(x, z_1)
    \frac{\pt}{\pt\nu_{z_2}}p^{n-1}_{0,t_2-t_1}(S^n_{j_{1}}z_1,z_2)\times
    \ldots\times \\
    & \phantom{aaaaaa} \times
     \frac{\pt}{\pt\nu_{z_{n-k}}}p^{k+1}_{0,t_{n-k}-t_{n-k-1}}(S^{k+2}_{j_{n-k-1}}z_{n-k-1},z_{n-k})  \times p^{k}_{0,t-t_{n-k}}(S^{k+1}_{j_{n-k}}z_{n-k},y),
\end{align*}
where $m$ is the surface measure on $\bigcup^{n-1}_{j=1}\pt \Delta_n, j,$ the operator $\frac{\pt}{\pt\nu_a}$ is the outward normal derivative w.r.t. the $a$-variables, and the mapping $S^m_j\colon \pt\Delta_{m,j}\to\Delta_{m-1}$ is given via
$$
S^m_j(u_1,\ldots, u_j, u_{j+1}, u_{j+2}, \ldots, u_m)=(u_1,\ldots, u_j,u_{j+2},\ldots, u_m), \quad j=\ov{1, m-1}, m\in\mbN.
$$
\end{thm}
The proof is standard and follows the ideas from \cite[Section 3]{17} (see also \cite[Section VII.5]{19}).

Recalling \eqref{eq:def.p.k.t.} note that each $p^k_t(U^{(n)};\cdot)$ satisfies  \eqref{eq2.0} with  $X_t$ replaced with $X_t^{U^{(n)}} = \{X(u^{(n)}_1, t), \ldots,$ $X(u^{(n)}_n, t)\},$ $n\in\mbN.$ 
\begin{thm}
\label{thm2.2}
For all $k\in\mbN$ \ $p^k_t(U^{(n)};\cdot)\nearrow {p}^k_t, n\to\infty,$ a.e..
\end{thm}
\begin{proof}
The restrictions imposed on $\{U^{(n)}\}_{n\geq1}$ imply that a.e.
$$
p^k_t(U^{(n)}; \cdot)\leq p^k_t(U^{(n+1)}; \cdot)<{p}^k_t, \quad n\in\mbN.
$$
Put $q(y)=\lim_{n\to\infty}p^k_t(U^{(n)}; y)$ a.e.. 
Given a bounded  continuous $f$ the dominated convergence theorem implies
\begin{align*}
\int_{\mbR^k}q(y)f(y)dy &=\lim_{n\to\infty}
\int_{\mbR^k}
p^k_t(U^{(n)}; y)f(y)dy= \\
&=\lim_{n\to\infty}
\E\sum_{\begin{subarray}{c}
u_1, \ldots, u_k\in X_t^{U^{(n)}}, \\
\mbox{\small all distinct}\end{subarray}}
f(u_1, \ldots, u_k) \sum_{i=k}^{n} \sum_{J\in\mathcal{J}_{n,n-i}} \1\left( J_t\left(U^{(n)}\right) = J \right)= \\ 
&=\lim_{n\to\infty}
\E\sum_{\begin{subarray}{c}
u_1, \ldots, u_k\in X_t^{U^{(n)}}, \\
\mbox{\small all distinct}\end{subarray}}
f(u_1, \ldots, u_k) 
\1\left(\left|X_t^{U^{(n)}}\right| \geq k\right)=\\ 
&=\int_{\mbR^k}{p}^k_t(y) f(y)dy,
\end{align*}
which proves the assertion.
\end{proof}

Theorem \ref{thm2.1} can be used to study the speed of convergence in Theorem \ref{thm2.2}. %Theorem \ref{thm2.3} illustrates this.
\begin{proof}[Proof of Theorem \ref{thm2.3}]
 Let $A_{\ve}=[x; x+\ve]$ for some $
x\in\mbR$ and any $\ve\ll 1.$ Consider
$$
0\leq\int_{A_{\ve}}\left({p}^1_t(y)-p^1_t(U^{(n)};y) \right)dy=\E\sum_{u\in X_t}\1_{A_{\ve}}(u)-\E\sum_{u\in X^{U^{(n)}}_t}\1_{A_{\ve}}(u).
$$
Using the reasoning of \cite[Appendix B]{15} one shows the existence of a constant $C$ such that
\begin{align*}
&\left|\E\sum_{u\in X_t}
\1_{A_{\ve}}(u)-
\Prob\left(X_t\cap A_{\ve}\ne
\emptyset\right)\right|
\leq C{\ve}^2, \\
&\left|\E\sum_{x\in X^{U^{(n)}}_t}
\1_{A_{\ve}}(u)-
\Prob\left(X^{U^{(n)}}_t\cap A_{\ve}\ne
\emptyset\right)\right|
\leq C{\ve}^{2}.
\end{align*}
Therefore
\begin{equation}
\label{eq2.6}
\limsup_{\ve\to0+}\ve^{-1}
\int_{A_\ve}\left({p}^1_t(y)-p^1_t(U^{(n)};y)\right)dy=\limsup_{\ve\to0+}\ve^{-1}\left(\Prob\left(X_t\cap A_\ve\ne\emptyset\right)- \Prob\left(X^{U^{(n)}}_t\cap A_\ve\ne\emptyset\right)\right).
\end{equation}
Using the notion of the dual Brownian web $\{\wt{X}(u,t)\mid u\in\mbR, t\in [0;1]\}$ running backwards in time and the non-crossing property of it \cite[Section 2.2]{16} one has:
\begin{align}
\label{eq:3.diff.prob}
& \Prob\left(X_t\cap A_\ve\ne\emptyset\right)-\Prob\left(X^{U^{(n)}}_t\cap A_\ve\ne\emptyset\right) = \Prob\left(\forall j=\ov{1,n} \ X(u^{(n)}_j,t)\notin A_\ve, X_t\cap A_\ve\ne\emptyset\right)\leq \nonumber  \\
& \phantom{aaaaaa} \leq \Prob\left(\wt{X}(x+\ve, t)\ne\wt{X}(x, t), \exists \ j\in\{1, \ldots, n-1\}\colon \right.
\nonumber \\
& \phantom{aaaaaa}\phantom{aaaaaa} \left. \left(\wt{X}(x, t);\wt{X}(x+\ve, t)\right)\subset \left(u^{(n)}_j; u^{(n)}_{j+1}\right)\right) \leq  \nonumber  \\
& \phantom{aaaaaa} \leq \E \1\left( X(x+\ve, t) - X(x,t) \leq \max_{
j=\ov{1, n-1}}\left(u^{(n)}_{j+1}-u^{(n)}_j\right) \right) \1\left(J_t((x,x+\ve)) = \emptyset\right) =  \nonumber  \\ 
&\phantom{aaaaaa}= \int_{\mbR^2} \1 \left( y_2-y_1<\max_{
j=\ov{1, n-1}}\left(u^{(n)}_{j+1}-u^{(n)}_j\right)
\right)
p^{\emptyset, k}_t\left((x, x+\ve); (y_1, y_2)\right)dy_1 dy_2,
\end{align}
since $X$ and $\wt{X}$ have the same distribution.
Here 
$$
p^{\emptyset, 2}_t(a;b)=p^2_{0,t}(a;b)=\frac{1}{2\pi t}e^{-\frac{\|a-b\|^2}{2t}}(1-e^{-(b_2-b_1)(a_2-a_1)}),
$$
thus there exists $C >0$ such that if $(y_1, y_2)\in\Delta_2, y_2-y_1\leq \delta_n,$ where $\delta_n=\max_{j=\ov{1, n-1}}(u^{(n)}_{j+1}-u^{(n)}_j),$ then
$$
p^{\emptyset, 2}_t((x; x+\ve); (y_1, y_2))\leq C g_t(x - y_1)\times\ve\delta_n.
$$
Substituting the last estimate into \eqref{eq:3.diff.prob} and returning to \eqref{eq2.6} we have:
$$
\limsup_{\ve\to0+}\ve^{-1}\int_{A_\ve}({p}^1_t(y)-p^1_t(U^{(n)};y))dy\leq
 C\int_{\mbR}dy_1\int^{y_1+\delta_n}_{y_1}\!\!\!\!  dy_2\ \! g_t(x -y_1)\delta_n\leq C\delta^2_n
$$
for new $C.$ The application of the Lebesque differentiation theorem completes the proof.
\end{proof}
\section{Acknowledgments}
The authors are very grateful to the anonymous referee for valuable comments and suggestions.% that greatly improved the paper. 

\end{document}